\documentclass[12pt,a4paper,dvipsnames]{amsart}

\usepackage[utf8]{inputenc}
\usepackage{enumerate,titletoc,a4}
\usepackage{amssymb,bbm}
\usepackage{tikz}
\usetikzlibrary{arrows}
\usepackage[titletoc]{appendix}
\usepackage[pagebackref,colorlinks,linkcolor=BrickRed,citecolor=OliveGreen,urlcolor=blue,hypertexnames=true]{hyperref}

\DeclareMathOperator{\id}{id}

\DeclareMathOperator{\GL}{GL}
\DeclareMathOperator{\Mat}{Mat}
\DeclareMathOperator{\mps}{mp}
\DeclareMathOperator{\bfs}{bf}
\DeclareMathOperator{\cent}{Cent}

\DeclareMathOperator{\dom}{dom}
\newcommand{\der}{\mathrm d}
\newcommand{\N}{\mathbb{N}}
\newcommand{\bC}{\bar{C}}
\newcommand{\kk}{\mathbbm k}
\newcommand{\rr}{\mathbbm r}
\newcommand{\rs}{\mathbbm s}
\newcommand{\MB}{{\bf M}}
\newcommand{\SB}{{\bf S}}
\newcommand{\rrb}{{\bf r}}
\newcommand{\rsb}{{\bf s}}
\newcommand{\Z}{\mathbb{Z}}

\newcommand{\RS}{\mathbb{S}}

\newcommand{\cA}{\mathcal{A}}
\newcommand{\cB}{\mathcal{B}}

\newcommand{\cE}{\mathcal{E}}
\newcommand{\cF}{\mathcal{F}}

\newcommand{\cM}{\mathcal{M}}
\newcommand{\cN}{\mathcal{N}}
\newcommand{\cO}{\mathcal{O}}

\newcommand{\cR}{\mathcal{R}}
\newcommand{\cU}{\mathcal{U}}
\newcommand{\cV}{\mathcal{V}}

\newcommand{\cS}{\mathcal{S}}

\newcommand{\bb}{\mathbf{b}}
\newcommand{\cc}{\mathbf{c}}

\newcommand{\iimath}{k}
\newcommand{\jjmath}{\ell}

\newcommand{\eqcolon}{\mathrel{\mathop:}=}
\newcommand{\ug}{\mathrm{g}}
\newcommand{\mdom}[1][]{\operatorname{dom}^{\operatorname{mp}}_{#1}}

\newcommand*{\mpsymbol}{
	\!\protect\raisebox{0.534ex}{\protect\scalebox{.28}{
			\begin{tikzpicture}[line width=0.6ex]
			\draw[arrows={<->}]
			(0,0) -- (5ex,0);
			\end{tikzpicture}
		}}\!
	}

\newcommand*{\mpsymboltikz}{
	\!\protect\raisebox{0.534ex}{\protect\scalebox{.28}{
			\begin{tikzpicture}[line width=0.6ex]
			\draw[arrows={<->}]
			(0,0) -- (5ex,0);
			\end{tikzpicture}
		}}\!
	}
	
\newcommand*{\bfsymbol}{
	\!\protect\raisebox{0.534ex}{\protect\scalebox{.28}{
			\begin{tikzpicture}[line width=0.6ex]
			\draw[loosely dotted, dash phase=-0.22ex, arrows={<->}]
			(0,0) -- (5ex,0);
			\end{tikzpicture}
		}}\!
	}

\newcommand{\cupn}{\!\cup\!}

\newcommand{\kotimes}{\otimes_{\kk}}
\newcommand{\Cotimes}{\otimes_C}
\newcommand{\bkotimes}[2]{\bigotimes_{#1}^{#2}{\mkern-3mu}{}_{\kk}\,}
\newcommand{\bCotimes}[2]{\bigotimes_{#1}^{#2}{\mkern-3mu}{}_C\,}
\newcommand{\bCotimesi}{(\bigotimes_{\!C})}
\newcommand{\bbCotimes}[2]{\bigotimes_{#1}^{#2}{\mkern-3mu}{}_{\bC}\,}
\newcommand{\bbCotimesi}{(\bigotimes_{\!\bC})}
\newcommand{\seqn}{n_1,\dots,n_G}

\newcommand{\udn}{\operatorname{UD}_{\seqn}(x)}
\newcommand{\ud}[1]{\operatorname{UD}_{#1}}
\newcommand{\mm}[1]{\operatorname{Mat}_{#1}(\kk)}
\newcommand{\gm}[1]{\operatorname{GM}_{#1}}

\newcommand{\Langle}{\!\mathop{<}\!}
\newcommand{\Rangle}{\!\mathop{>}}
\makeatletter
\def\moverlay{\mathpalette\mov@rlay}
\def\mov@rlay#1#2{\leavevmode\vtop{
    \baselineskip\z@skip \lineskiplimit-\maxdimen
    \ialign{\hfil$#1##$\hfil\cr#2\crcr}}}
\makeatother

\newcommand{\plangle}{\moverlay{(\cr<}}
\newcommand{\prangle}{\moverlay{)\cr>}}
\newcommand{\px}[1]{\kk\Langle X^{(#1)} \Rangle}
\newcommand{\ppx}{\kk\Langle X \Rangle}
\newcommand{\rx}[1]{\kk\plangle X^{(#1)} \prangle}
\newcommand{\re}{\cR_{\kk}(X)}

\newcommand{\pz}{D\Langle Z \Rangle}
\newcommand{\rz}{D\plangle Z \prangle}
\newcommand{\rez}{\cR_D(Z)}
\newcommand{\rezz}{\cR_{\kk}(Z)}

\newcommand{\mpvar}[1]{\mathcal{MPV}^{\ug}(#1)}
\newcommand{\prempr}[2]{\kk\plangle X^{(#1)} \protect\mpsymbol \cdots \protect\mpsymbol X^{(#2)}\prangle}
\newcommand{\mpr}{\protect\prempr{1}{G}}
\newcommand{\pretxx}[2]{\kk\Langle X^{(#1)} \mpsymbol \cdots \mpsymbol X^{(#2)}\Rangle}
\newcommand{\txx}{\pretxx{1}{G}}

\newcommand{\pxy}{\kk\Langle X\cup Y \Rangle}
\newcommand{\bfa}{\kk\Langle X \bfsymbol Y \Rangle}
\newcommand{\bfsf}{\kk\plangle X \bfsymbol Y\prangle}
\newcommand{\rexy}{\cR_{\kk}(X\cup Y)}
\newcommand{\bfvar}[1]{\mathcal{BFV}^g_{#1}}

\newtheorem{thm}{Theorem}[section]
\newtheorem{lem}[thm]{Lemma}
\newtheorem{cor}[thm]{Corollary}
\newtheorem{prop}[thm]{Proposition}
\theoremstyle{definition}
\newtheorem{defn}[thm]{Definition}

\theoremstyle{remark}
\newtheorem{rem}[thm]{Remark}

\numberwithin{equation}{section}

\begin{document}

\setcounter{tocdepth}{3}
\contentsmargin{2.55em} 
\dottedcontents{section}[3.8em]{}{2.3em}{.4pc} 
\dottedcontents{subsection}[6.1em]{}{3.2em}{.4pc}
\dottedcontents{subsubsection}[8.4em]{}{4.1em}{.4pc}

\makeatletter
\newcommand{\mycontentsbox}{%
{\centerline{NOT FOR PUBLICATION}
\addtolength{\parskip}{-2.3pt}
\tableofcontents}}
\def\enddoc@text{\ifx\@empty\@translators \else\@settranslators\fi
\ifx\@empty\addresses \else\@setaddresses\fi
\newpage\mycontentsbox\newpage\printindex}
\makeatother

\setcounter{page}{1}

\title{Multipartite rational functions}

\author[Igor Klep]{Igor Klep${}^1$}
\address{Igor Klep, University of Ljubljana, Faculty of Mathematics and Physics}%, Department of Mathematics}
\email{igor.klep@fmf.uni-lj.si}
\thanks{${}^1$Supported by the Slovenian Research Agency grants J1-8132, N1-0057 and P1-0222. Partially supported by the Marsden Fund Council of the Royal Society of New Zealand.}

\author[Victor Vinnikov]{Victor Vinnikov${}^2$}
\address{Victor Vinnikov, 
Ben-Gurion University of the Negev, Department of Mathematics}
    \email{vinnikov@math.bgu.ac.il}
    \thanks{${}^2$Supported by the Deutsche Forschungsgemeinschaft (DFG) Grant No. SCHW 1723/1-1.}

\author[Jurij Vol\v{c}i\v{c}]{Jurij Vol\v{c}i\v{c}${}^3$}
\address{Jurij Vol\v{c}i\v{c}, Texas A\&M University, Department of Mathematics}
    \email{volcic@math.tamu.edu}
\thanks{${}^3$Supported by the NSF grant DMS 1954709. Partially supported by the University of Auckland Doctoral Scholarship and by the Deutsche Forschungsgemeinschaft (DFG) Grant No. SCHW 1723/1-1.}

\date{\today}

\begin{abstract}
Consider a tensor product of free algebras over a  field $\kk$, the so-called multipartite free algebra $\cA=\px{1}\otimes\cdots\otimes\px{G}$. It is well-known that $\cA$ is a domain, but not a fir nor even a Sylvester domain. Inspired by recent advances in free analysis, formal rational expressions over $\cA$ together with their matrix representations in $\mm{n_1}\otimes\cdots\otimes\mm{n_G}$ are employed to construct a skew field of fractions $\cU$ of $\cA$, whose elements are called multipartite rational functions. It is shown that $\cU$ is the universal skew field of fractions of $\cA$ in the sense of Cohn. As a consequence a multipartite analog of Amitsur's theorem on rational identities relating evaluations in matrices over $\kk$ to evaluations in skew fields is obtained. The characterization of $\cU$ in terms of matrix evaluations fits naturally into the wider context of free noncommutative function theory, where multipartite rational functions are interpreted as higher order noncommutative rational functions with an associated difference-differential calculus and linear realization theory. Along the way an explicit construction of the universal skew field of fractions of $D\otimes \ppx$ for an arbitrary skew field $D$ is given using matrix evaluations and formal rational expressions. 
\end{abstract}

%%%%%%% here is the text version of the abstract - use on arxiv
\iffalse
Consider a tensor product of free algebras over a  field $k$, the so-called multipartite free algebra $A=k \langle X^{(1)}\rangle\otimes\cdots\otimes k\langle X^{(G)}\rangle$. It is well-known that $A$ is a domain, but not a fir nor even a Sylvester domain. Inspired by recent advances in free analysis, formal rational expressions over $A$ together with their matrix representations in $M_{n_1}(k)\otimes\cdots\otimes M_{n_G}(k)$ are employed to construct a skew field of fractions $U$ of $A$, whose elements are called multipartite rational functions. It is shown that $U$ is the universal skew field of fractions of $A$ in the sense of Cohn. As a consequence a multipartite analog of Amitsur's theorem on rational identities relating evaluations in matrices over $k$ to evaluations in skew fields is obtained. The characterization of $U$ in terms of matrix evaluations fits naturally into the wider context of free noncommutative function theory, where multipartite rational functions are interpreted as higher order noncommutative rational functions with an associated difference-differential calculus and linear realization theory. Along the way an explicit construction of the universal skew field of fractions of $D\otimes k\langle X\rangle$ for an arbitrary skew field $D$ is given using matrix evaluations and formal rational expressions. 
\fi
%%%%%%%

\maketitle

%%%%%%%%%%%%%%%%%%%%%%%%%%%%%%%%%%%%%%%%%%%%%%%%%%%%%%%%%%%%%%%%%%%%%
%%%%%%%%%%%%%%%%%%%%%%%%%%%%%%%%%%%%%%%%%%%%%%%%%%%%%%%%%%%%%%%%%%%%%
%%%%%%%%%%%%%%%%%%%%%%%%%%%%%%%%%%%%%%%%%%%%%%%%%%%%%%%%%%%%%%%%%%%%%
%%%%%%%%%%%%%%%%%%%%%%%%%%%%%%%%%%%%%%%%%%%%%%%%%%%%%%%%%%%%%%%%%%%%%

\section{Introduction}

The question of embeddability of a noncommutative ring into a skew field is much more complex than its counterpart in the commutative setting. The classical construction of a field of fractions extends beyond commutative rings in a straightforward way only to Ore domains \cite[Section 2.1]{MR}: every left (resp.~right) Ore domain admits a left (resp.~right) classical ring of quotients, whose elements are of the form $a^{-1}b$ (resp.~$a^{-1}b$). However, in general not only is there no simple criterion for the existence of a skew field of fractions \cite[Section 6.7]{Co1}, even if one exists it is not necessarily unique \cite{Fi}. It is therefore natural to ask whether there exists a skew field of fractions of a given ring that is the largest possible in some sense. Cohn made this notion precise by introducing the \emph{universal skew field of fractions} of a ring \cite[Section 7.2]{Co3}: if $R$ is a ring and $\cU$ is its skew field of fractions, then $\cU$ is called universal if every epimorphism from $R$ to a skew field $D$ extends to a specialization from $\cU$ to $D$.

Well-known examples of rings admitting universal skew fields of fractions are Sylvester domains \cite[Sections 5.5 and 7.5]{Co3}, and among them firs (free ideal rings) and semifirs. If $R$ is a Sylvester domain, then the localization with respect to full matrices over $R$ yields a universal skew field of fractions of $R$ by \cite[Chapter 4]{Co1}. However, the elements of this construct often lack simple canonical forms. In some special cases one can find more explicit descriptions; for instance, a free algebra over a field is a fir and \cite{Ami,Le,Li,HMV} provide different constructions of the free skew field, i.e., its universal skew field of fractions. 

Apart from the aforementioned family, only isolated examples of rings admitting a universal skew field of fractions are known. We now proceed to describe a new class with this property. Let $\kk$ be a field of characteristic 0 and $G\in \N$. By $\txx$ we denote the tensor product of free $\kk$-algebras over the sets $X^{(1)},\dots, X^{(G)}$, which we call a \emph{multipartite free $\kk$-algebra}. This terminology alludes to bi- and multipartite systems of operators arising in quantum theory \cite{Pe,H} and free probability \cite{Voi1}; another source of multipartite free variables are trace monoids in automata theory \cite{DK,Wor}. Homological properties of these rings and their generalizations (free partially commutative algebras) are studied in \cite{DL}. While it is easy to see that a multipartite free algebra is a domain, it satisfies the Ore condition if and only if $|X^{(1)}|=\cdots=|X^{(G)}|=1$ and it is a Sylvester domain if and only if $G\le 2$ and $|X^{(1)}|=1$ or $|X^{(2)}|=1$ by \cite[Theorem 3.1]{Co2}. Cohn was able to prove that $\kk\Langle X^{(1)} \mpsymbol X^{(2)}\Rangle$ has a universal field of fractions in \cite[Theorem 3.1]{Co2}. His method  relied heavily on the condition $G=2$ and does not generalize to $\txx$ for $G>2$.

The main result of this paper is the following.

\begin{thm}\label{t:intro}
Let $G\in\N$ be arbitrary. Then $\txx$ admits a universal skew field of fractions.
\end{thm}

The proof of Theorem \ref{t:intro} consists of two parts: first we introduce the set of \emph{multipartite rational functions}, denoted $\mpr$, and prove it is a skew field containing the multipartite free algebra $\txx$ (Theorem \ref{t:div}). We then in Theorem \ref{t:univ} establish its universal property. The construction of $\mpr$ is inspired by the description of noncommutative rational functions \cite{HMV,KVV1} as the equivalence classes of formal rational expressions with respect to their evaluations on matrix tuples. In our setting, we consider multipartite evaluations of expressions; these are defined using Kronecker's tensor product of matrices, which models the commutativity relations among variables in a multipartite free algebra. The main intermediate step towards universality is a new construction of the universal skew field of fractions of $D\otimes \ppx$ for an arbitrary skew field $D$ based on matrix evaluations.

The paper is organized as follows. In Section \ref{sec2} we introduce the necessary terminology and the notion of multipartite generic matrices. Section \ref{sec3} starts with the definition of multipartite rational functions and in Theorem \ref{t:div} we prove that $\mpr$ is a skew field of fractions of $\txx$; some traits of its internal structure are described in Subsection \ref{subsec32}.  In the first part of Section \ref{sec4} we develop auxiliary results that lead to Corollary \ref{c:rz}, a new characterization of the universal skew field of fractions for the tensor product of a skew field and a free algebra. The main result of the paper is Theorem \ref{t:univ}, where we prove the universality of $\mpr$. The proof uses generic matrices and PI-theory techniques as well as Cohn's results on localization of (semi)firs; the key connecting element between these methods is the block structure of multipartite evaluations. Akin to Amitsur's theorem \cite[Theorem 16]{Ami}, Theorem \ref{t:sci} shows that a rational expression vanishes on all tuples of matrices over $\kk$ satisfying commutation relations imposed by the multipartite free algebra if and only if it vanishes on all tuples over skew fields satisfying the same commutation relations.

In Section \ref{sec5} we place multipartite rational functions in the context of free function theory \cite{Voi0,KVV3,HKM}, where they play the role of higher order noncommutative rational functions in the sense of \cite[Chapter 3]{KVV3}. In Subsection \ref{subsec52} we briefly describe their difference-differential calculus, while in Subsection \ref{subsec53} we discuss their matrix coefficient realizations. Given a multipartite rational function $\rrb$, its minimal size realization can be regarded as a normal form for $\rrb$. Lastly, in Appendix \ref{appA} we provide a matrix model for the skew field of \emph{bi-free rational functions}, a notion that is motivated by recent progress in free probability \cite{Voi1} and is closely related to multipartite rational functions championed in this article.

\subsection*{Acknowledgments}

The authors thank Roland Speicher for drawing the free probability aspect of multipartite rational functions to their attention.

\linespread{1.05}

\section{Preliminaries}\label{sec2}

In this section we gather some preliminaries and background that will be used throughout the paper. This includes some notions from skew fields \cite{Co1,Co3} and the theory of polynomial identities \cite{Row}.

\subsection{Notation and terminology}\label{subsec21}

Throughout the paper let $\kk$ be a fixed (commutative) field of characteristic 0. We assume that all rings have a multiplicative identity and that the latter is preserved under ring homomorphisms. The tensor product over a ring $R$ is denoted $\otimes_R$.

\begin{defn}[{\cite[Section 4.1]{Co1} or \cite[Section 7.2]{Co3}}]
If $F$ and $E$ are skew fields, then a {\bf local homomorphism} (or a {\bf subhomomorphism}) $\lambda:F\dashrightarrow E$ is given by a ring homomorphism $F_0\to E$, whose domain $F_0\subseteq F$ is a local subring and in whose kernel are precisely the elements that are not invertible in $F_0$.
\end{defn}

\begin{defn}[{\cite[Section 7.2]{Co3}}]\label{d:sff}
Let $R$ be a ring. A skew field $\cU$ is a {\bf skew field of fractions} of $R$ if there is an embedding $R\hookrightarrow \cU$ and its image generates $\cU$ as a skew field. If furthermore every homomorphism $R\to D$, where $D$ is a skew field, extends to a local homomorphism $\cU\dashrightarrow D$, whose domain contains $R$, then $\cU$ is called the {\bf universal skew field of fractions} of $R$.
\end{defn}

\begin{rem}\label{r:another}\mbox{}\par\vspace{-2mm}
\begin{enumerate}
\item The universal skew field of fractions is, when it exists, unique up to isomorphism \cite[Section 7.2]{Co3}.
\item An alternative characterization of the universal skew field of fractions $\cU$ of $R$ is as follows (see \cite[Theorem 7.2.7]{Co3}): every matrix over $R$, which becomes invertible under some homomorphism from $R$ to a skew field, is invertible over $\cU$.
\end{enumerate}
\end{rem}

Let $G\in\N$ and $n_i\in\N$ for $1\le i\le G$. Recall that
\begin{equation}\label{e:can}
\bkotimes{i=1}{G} \mm{n_i} \to \mm{n_1\cdots n_G},\qquad \bkotimes{i=1}{G} a_i\mapsto \bigotimes_{i=1}^G a_i,
\end{equation}
 where $\otimes$ is the Kronecker product of matrices, is an isomorphism of $\kk$-algebras. Furthermore, for every permutation $\pi$ of the set $\{1,\dots, G\}$ there exists a permutation matrix $K_{\pi;\seqn}\in \GL_{n_1\cdots n_G}(\kk)$, called the {\bf commutation matrix}, such that
\begin{equation}\label{e:tenspi}
\bkotimes{i=1}{G} a_{\pi(i)}=K_{\pi;\seqn}\left(\bkotimes{i=1}{G} a_i\right)K_{\pi;\seqn}^t
\end{equation}
for all $a_i\in\mm{n_i}$. Let $\tau_i: \mm{n_i}\hookrightarrow \mm{n_1\cdots n_G}$ denote the embedding corresponding to the isomorphism \eqref{e:can}, i.e.,
\begin{equation}\label{e:tau0}
\tau_i(a)=I_{n_1}\otimes\cdots\otimes a\otimes \cdots\otimes I_{n_G}.
\end{equation}

For $1\le i\le G$ let $X^{(i)}=\{X_1^{(i)},\dots, X_{g_i}^{(i)}\}$ be sets of freely noncommuting variables and set $X=\bigcup_i X^{(i)}$. The principal object of this paper is the {\bf multipartite free $\kk$-algebra}
\begin{align*}
\kk \Langle X^{(1)} \mpsymbol X^{(2)} \mpsymbol\cdots\mpsymbol X^{(G)} \Rangle 
\eqcolon&\ \px{1} \kotimes \px{2} \kotimes\cdots\kotimes \px{G} \\
\cong&\ \ppx \big/ \left([X_{j_1}^{(i_1)},X_{j_2}^{(i_2)}] \colon i_1\neq i_2,\ j_1,j_2 \right).
\end{align*}
Let
$$\kk[\zeta^{(i)}]=k[\zeta_{j\iimath\jjmath}^{(i)}:1\le j\le g_i, 1\le \iimath,\jjmath \le n_i],\qquad \kk[\zeta]=\bkotimes{i=1}{G}\kk[\zeta^{(i)}];$$
corresponding fields of fractions are $\kk(\zeta^{(i)})$ and $\kk(\zeta)$, respectively. Clearly, the map
\begin{equation}\label{e:tau1}
\bkotimes{i=1}{G} \Mat_{n_i}(\kk(\zeta^{(i)}))\to \Mat_{n_1\cdots n_G}(\kk(\zeta)), \qquad
\bkotimes{i=1}{G} c_i\mapsto \bigotimes_{i=1}^G c_i
\end{equation}
is an embedding. For $u\in \Mat_n(\kk(\xi_1,\dots, \xi_m))$ let $\dom u\subseteq \kk^{m n^2}$ denote the intersection of domains of its entries.

Furthermore, let $\gm{n_i}(x^{(i)})$ be the {\bf algebra of generic matrices}, i.e., the unital $\kk$-subalgebra of $\Mat_{n_i}(\kk[\zeta^{(i)}])$ generated by $g_i$ matrices $x_j^{(i)}=(\zeta_{j\iimath\jjmath}^{(i)})_{\iimath\jjmath}$ of size $n_i$. We refer to \cite[Section 1.3]{Row} for its role in the theory of polynomial identities; also see \cite{Pro} for a more geometric interpretation. Lastly, let
$$\gm{\seqn}(x)=\bkotimes{i=1}{G}\gm{n_i}(x^{(i)}) \hookrightarrow \Mat_{n_1\cdots n_G}(\kk[\zeta])$$
be the {\bf algebra of multipartite generic matrices}, where the last map is the restriction of the embedding \eqref{e:tau1}. Images of the generic matrices $x^{(i)}_j$ in $\gm{\seqn}(x)$ are called {\bf multipartite generic matrices.}

Finally, for the sake of simplicity we adopt the following phrasing conventions. Let $r$ be a mapping that is defined on $\cS_0\subseteq\cS$, not defined on $\cS\setminus\cS_0$ and has 0 in its codomain; if $\cS_0\neq \emptyset$ and $r|_{S_0}=0$, then we say that $r$ {\bf vanishes} on $\cS$. If $\cV$ is an affine variety over $\kk$ and every point in a non-empty Zariski-open subset of $\cV$ satisfies a property $P$, then we say that {\bf almost every} point of $\cV$ satisfies $P$. Normally one would say that $P$ is generically true on $\cV$, but to avoid confusion with generic matrices, which are frequently used in this paper, we prefer to adapt the non-standard notion.

\subsection{Local results}\label{subsec22}

The following proposition describes the universal property of the $\kk$-algebra $\gm{\seqn}(x)$ with respect to central simple algebras. Let $\cA_i$ be unital $\kk$-algebras with common central subfield $C\supseteq \kk$. A {\bf tensor evaluation} in $\bCotimesi_i\cA_i$ is a homomorphism $\varphi:\txx\to \bCotimesi_i\cA_i$ satisfying $\varphi(\px{i})\subseteq \cA_i$.

\begin{prop}\label{p:tens}
Let $\cA_i$ be simple algebras of degrees $n_i$ with common center $C\supseteq \kk$. Let $p\in\txx$ be arbitrary. Then $p$ vanishes under every tensor evaluation in $\bCotimesi_i\cA_i$ if and only if the canonical image of $p$ in $\gm{\seqn}(x)$ is zero.
\end{prop}

\begin{proof}
For every tensor evaluation in $\bCotimesi_i\cA_i$ we have a sequence of homomorphisms
$$\gm{n_1,\dots,n_G}(x)\to \bkotimes{i=1}{G} \cA_i
\to \bCotimes{i=1}{G}\cA_i \to \bbCotimes{i=1}{G}(\bC\Cotimes\cA_i)
\to \bbCotimes{i=1}{G}\Mat_{n_i}(\bC),$$
where $\bC$ is the algebraic closure of $C$. The first homomorphism exists since $\gm{n_i}(x^{(i)})$ are relatively free (\cite[Lemma 14.1]{Sa} or \cite[Proposition 1.3.9, Theorem 1.3.11]{Row}) and the last homomorphism (in fact an isomorphism) simply states that $\bC$ is a splitting field for $\cA_i$. Note that all homomorphisms are either surjective or they correspond to central extensions. Therefore if $p$ vanishes in $\gm{\seqn}(x)$, it vanishes under every tensor evaluation in $\bCotimesi_i\cA_i$; and if $p$ vanishes under every tensor evaluation in $\bCotimesi_i\cA_i$, it vanishes under every tensor evaluation in $\bbCotimesi_i\Mat_{n_i}(\bC)$. Because $\bC$ is infinite, the canonical image of $p$ in $\gm{\seqn}(x)$ is zero if and only if $p$ vanishes under every tensor evaluation in $\bbCotimesi_i\Mat_{n_i}(\bC)$, hence the statement is proved.
\end{proof}

\begin{lem}\label{l:prime}
$\gm{\seqn}(x)$ is a prime ring.
\end{lem}

\begin{proof}
Let $0\le H\le G$ and without loss of generality assume $g_i>1$ for $i\le H$ and $g_i=1$ for $i> H$. Firstly, it is clear that $\gm{n_i}(x^{(i)})\cong \kk[t_i]$ for $i>H$, where $t_i$ is an auxiliary symbol. Hence we have
$$\gm{\seqn}(x)\cong \kk[t_{H+1},\dots,t_G]\kotimes 
\bkotimes{i=1}{H} \gm{n_i}(x^{(i)}) \hookrightarrow
\Mat_{n_1\cdots n_{H}}(\kk(t,\zeta')),
$$
where $t=\{t_{H+1},\dots,t_G\}$ and $\zeta'=\bigcup_{i=1}^{H}\zeta^{(i)}$. By \cite[Remark 1.9.5]{Row} it is enough to prove that $\Mat_{n_1\cdots n_{H}}(\kk(t,\zeta'))$ is a central extension of $\gm{n_1,\dots,n_H}(x)$.

By \cite[Proposition 2.4.11]{Row}, $\Mat_{n_i}(\kk(\zeta^{(i)}))$ is a central extension of $\gm{n_i}(x^{(i)})$ for every $i\le H$, so there exists a basis $S_i\subset \gm{n_i}(x^{(i)})$ of $\Mat_{n_i}(\kk(\zeta^{(i)}))$ as an algebra over its center. Since $S_i$ is $\kk$-linearly independent under almost every evaluation, the set
$$S=\left\{\bkotimes{i=1}{H} s_i\colon s_i\in S_i\right\}$$
is also $\kk$-linearly independent under almost every evaluation. Therefore it is $\kk(t,\zeta')$-linearly independent in $\Mat_{n_1\cdots n_{H}}(\kk(t,\zeta'))$. Finally, $S$ is a basis for $\Mat_{n_1\cdots n_{H}}(\kk(t,\zeta'))$ since $|S|=n_1^2\cdots n_H^2$.
\end{proof}

Since $\gm{\seqn}(x)$ is a prime $\kk$-algebra, its center is an integral domain. Let $\udn$ be the ring of central quotients of $\gm{\seqn}(x)$; that is, $\udn$ is the localization of $\gm{\seqn}(x)$ at its nonzero central elements. Because the inclusion $\gm{n_i}(x^{(i)})\hookrightarrow \gm{\seqn}(x)$ preserves central elements, the map \eqref{e:tau1} restricts to an embedding
$$\bkotimes{i=1}{G} \ud{n_i}(x^{(i)})\hookrightarrow \udn.$$
Since $\gm{\seqn}(x)$ is a prime PI-ring, $\udn$ is a simple algebra of finite degree by \cite[Theorem 1.7.9]{Row}. This leads to the next proposition, which generalizes the $G=1$ case \cite[Theorem 3.2.6]{Row} (cf.~\cite[Proposition 2.1]{KVV2}).

\begin{prop}\label{p:ud}
$\udn$ is a skew field.
\end{prop}

\begin{proof}
Suppose $\udn$ is not a skew field. Then it contains nilpotents, so there exists a nonzero $p\in \gm{\seqn}(x)$ such that $p^2=0$. Therefore for every $\cA_i$ and $C$ as in Proposition \ref{p:tens} there exist $u^{(i)}_1,\dots,u^{(i)}_{g_i}\in \cA_i$ such that $p(u)\neq0$ and $p(u)^2=0$ in 
$\bCotimesi_i\cA_i$. By \cite[Proposition 1.1 and the preceding paragraph]{Sa}, we can find (cyclic) algebras 
$\cA_i$ such that $\bCotimesi_i\cA_i$ is a (crossed product) division algebra, which leads to a contradiction.
\end{proof}

\section{Skew field of multipartite rational functions}\label{sec3}

In this section we introduce multipartite rational functions using evaluations in matrix algebras of arbitrary size, and give some of their basic properties.

Set $\ug=(g_1,\dots,g_G)$ and let
$$\cM_{\seqn}^{\ug}=\prod_{i=1}^G \mm{n_i}^{g_i},\qquad \cM^{\ug}=\bigcup_{\seqn} \cM_{\seqn}^{\ug}.$$
Define $\tau:\cM^{\ug}\to\cM^{\ug}$ by
\begin{equation}\label{e:tau}
\tau(a^{(1)},\dots,a^{(G)})=(\tau_1(a^{(1)}),\dots, \tau_G(a^{(G)})),
\end{equation}
where $\tau_i$ are defined by \eqref{e:tau0}. Let $\re$ be the set of noncommutative rational expressions over $\kk$, i.e., all possible syntactically valid combinations of elements in $\kk$ and $X$, arithmetic operations (addition, multiplication, inversion) and parentheses. {\bf The inversion height} of $r\in\re$ is the maximum number of nested inverses in $r$. We can attempt to evaluate nc rational expressions on tuples of square matrices of the same size; such an evaluation will be occasionally  called {\bf nc-evaluation} to distinguish it from other types of evaluations. The set of all tuples of matrices (of size $n$), at which $r$ is defined, is denoted $\dom r$ (resp. $\dom_n r$) and called the {\bf domain} of $r$.

\begin{defn}
If $r\in \re$ is defined at $\tau(a)$ for $a\in \cM^{\ug}$, then we say that $r$ is {\bf mp-defined} at $a\in \cM^{\ug}$ and write $r(a)^{\mps}=r(\tau(a))$. An expression $r\in\re$ is {\bf mp-nondegenerate} if it is mp-defined somewhere on $\cM^{\ug}$. Its {\bf mp-domain} in $\cM_{\seqn}^{\ug}$ (resp. $\cM^{\ug}$) is denoted $\mdom[\seqn]r$ (resp. $\mdom r$).
\end{defn}

Note that mp-domains are Zariski-open sets and $\mdom[\seqn]r\subseteq \dom r(x)$, where $x$ is the tuple of multipartite generic matrices from $\gm{\seqn}(x)$. Basic properties of mp-evaluations are summarized in the following proposition (cf.~Subsection \ref{subsec51}).

\begin{prop}\label{p:basic}
Let $r\in\re$.
\begin{enumerate}[\rm(1)]
\item If $(a',a^{(2)},\dots,a^{(G)})\in\mdom[n_1',n_2,\dots,n_G]r$ and $(a'',a^{(2)},\dots,a^{(G)})\in\mdom[n_1'',n_2,\dots,n_G]r$, then 
$(a'\oplus a'',a^{(2)},\dots,a^{(G)}) \in \mdom[n_1'+n_1'',n_2,\dots,n_G] r$ and
$$r(a'\oplus a'',a^{(2)},\dots,a^{(G)})^{\mps}=
r(a',a^{(2)},\dots,a^{(G)})^{\mps}\oplus r(a'',a^{(2)},\dots,a^{(G)})^{\mps}.$$
\item If $(a^{(1)},\dots,a^{(G)})\in\mdom[\seqn]r$, then $(p_1a^{(1)}p_1^{-1},\dots, p_Ga^{(G)}p_G^{-1})\in\mdom[\seqn] r$ for all $p_i\in\GL_{n_i}(\kk)$ and
$$r(p_1a^{(1)}p_1^{-1},\dots, p_Ga^{(G)}p_G^{-1})^{\mps}=\left(\bigotimes_i p_i\right)r(a^{(1)},\dots,a^{(G)})^{\mps}\left(\bigotimes_i p_i\right)^{-1}.$$
\end{enumerate}
\end{prop}

\begin{proof}
Straightforward.
\end{proof}

\begin{rem}\label{r:updown}
Proposition \ref{p:basic}(1) and \eqref{e:tenspi} imply
$$\mdom[\seqn]r\neq\emptyset \Rightarrow \mdom[k_1n_1,\dots,k_Gn_G]r\neq\emptyset$$
and
$$(\forall a\in\mdom[k_1n_1,\dots,k_Gn_G]r\colon r(a)^{\mps}=0) \Rightarrow 
(\forall a\in\mdom[\seqn]r\colon r(a)^{\mps}=0)$$
hold for all $n_1,\dots,n_G\in\N$ and $k_1,\dots,k_G\in\N$. These implications enable us to traverse ``up'' and ``down'' between the level sets $\cM_{\seqn}^{\ug}$ in $\cM^{\ug}$.
\end{rem}

\begin{lem}\label{l:det}
Let $r\in\re$. If $\det r$ mp-vanishes on $\cM^{\ug}_{\seqn}$, then $r$ mp-vanishes on $\cM^{\ug}_{\seqn}$.
\end{lem}

\begin{proof}
If $x$ is the tuple of multipartite generic matrices in $\gm{\seqn}(x)$, then $r(x)\in \udn$. By assumption we have $\det r(a)^{\mps}=0$ for all $a\in\mdom[\seqn]r$ and thus $\det r(x)=0$. Therefore $r(x)$ is a zero divisor in $\udn$, but the latter is a skew field by Proposition \ref{p:ud}, so $r(x)=0$ and hence $r(a)^{\mps}=0$ for all $a\in\mdom[\seqn]r$.
\end{proof}

Consider the relation
\begin{equation}\label{e:rel}
r_1\sim^{\mps} r_2 \iff r_1(a)^{\mps}=r_2(a)^{\mps}\ \ \forall a\in\mdom r_1\cap\mdom r_2
\end{equation}
on the set of all mp-nondegenerate expressions in $\re$. It is not hard to check that $\sim^{\mps}$ is an equivalence relation; transitivity is proved using Remark \ref{r:updown} and the fact that the set $\mdom[\seqn]r$ is Zariski-open in $\cM^{\ug}_{\seqn}$ for every $r\in\re$. Let $\mpr$ be the set of equivalence classes of mp-nondegenerate expressions with respect to $\sim^{\mps}$. It becomes a $\kk$-algebra when endowed with the natural addition and multiplication. The equivalence class of $r\in\re$ is denoted $\rrb\in\mpr$. The domain of $\rrb$ is defined as the union of mp-domains of all representatives of $\rrb$ and is denoted $\dom \rrb$; the evaluation of $\rrb$ at $a\in\cM^{\ug}$ is then $\rrb(a)=r(a)^{\mps}$ for any representative $r\in\re$ such that $a\in\mdom r$. Elements of $\mpr$ are called {\bf multipartite (mp) rational functions}.

If $G=1$, our construction recovers the skew field of {\bf noncommutative (nc) rational functions}, see \cite{HMV,KVV1,KVV2}. In this case the equivalence class of a nondegenerate expression $r\in\cR_{\kk}(X^{(1)})$ is more commonly denoted $\rr\in \rx{1}$.

\begin{rem}\label{r:nn}
It is easy to see that we would get the same equivalence relation $\sim^{\mps}$ if we considered only mp-evaluations on $\bigcup_n \cM^{\ug}_{n,\dots,n}$ instead of $\cM^{\ug}$.
\end{rem}

\begin{rem}\label{r:shuffle}
	Let $\pi$ be a permutation on the set $\{1,\dots,G\}$. Then \eqref{e:tenspi} implies that there is an isomorphism
	$$\Psi_{\pi}\colon\mpr\to \prempr{\pi(1)}{\pi(G)}$$
	satisfying
	$$\Psi_{\pi}(\rrb)(a)=K_{\pi} \rrb(a)K_{\pi}^t,$$
	i.e., evaluations $\rrb(a)$ and $\Psi_{\pi}(\rrb)(a)$ are equal up to conjugation with a commutation matrix. Moreover,
	$$\kk\plangle Y^{(1)} \mpsymbol \cdots \mpsymbol Y^{(G)} \prangle\subseteq \mpr$$
	for $Y^{(i)}\subseteq X^{(i)}$.
\end{rem}

\begin{thm}\label{t:div}
The $\kk$-algebra $\mpr$ is a skew field and the $\kk$-algebra $\prempr{1}{G-1}\kotimes \rx{G}$ naturally embeds into $\mpr$. Therefore $\mpr$ is a skew field of fractions of $\txx$.
\end{thm}

\begin{proof}
The first statement is a direct consequence of Lemma \ref{l:det}. For the second statement, observe that $\prempr{1}{G-1}$ and $\rx{G}$ embed into $\mpr$, so we have a homomorphism
$$\prempr{1}{G-1}\kotimes \rx{G}\to\mpr.$$
Assume $\sum_i\rrb_i \otimes \rs_i$ lies in its kernel and $\rs_i$ are $\kk$-linearly independent in $\rx{G}$. By the local-global principle for linear dependence of nc rational functions as in \cite[Corollary 8.87]{HKM} or \cite[Theorem 6.6]{Vol}, there exists $b\in \mm{n_G}^{g_G}$ such that $\rs_i(b)$ are $\kk$-linearly independent. Since
$$\sum_i\rrb_i(a) \otimes \rs_i(b)=\left(\sum_i\rrb_i\rsb_i\right)(a,b)=0$$
for almost every $a\in \cM^{\ug}$, we have $\rrb_i(a)=0$ for all $a$ by the property of the tensor product and so $\rrb_i=0$.
\end{proof}

\subsection{Matrices over mp rational functions}\label{subsec31}

The next local-global property will be crucial in the proof of the main result in 
Subsection \ref{subsec42}.

\begin{prop}\label{p:inv}
Let $\MB$ be a $d\times d$ matrix over $\mpr$. Then $\MB$ is invertible if and only if $\MB(a)\in \mm{d n_1\cdots n_G}$ is invertible for some $a\in\cM^{\ug}_{\seqn}$.
\end{prop}

\begin{proof}
If $\MB$ is invertible, then the intersection of the domains of entries of $\MB$ and $\MB^{-1}$ is non-empty, so it contains some $a\in\cM^{\ug}$; then $\MB(a)$ is obviously invertible.

The converse is proved by induction on $d$; the basis of induction $d=1$ is clear by definition of $\mpr$, so let $d>1$. Assume there exists $a\in\cM^{\ug}_{\seqn}$ such that all entries of $\MB$ are defined in $a$ and $\MB(a)$ are invertible. Obviously, this then holds for almost every $a\in\cM^{\ug}_{\seqn}$. Clearly $\MB$ has at least one nonzero entry; after permuting rows and columns we can assume that $\MB_{11}\neq0$. By the previous argument there exists $a'\in\cM^{\ug}_{\seqn}$ such that $\MB(a')$ and $\MB_{11}(a')$ are invertible. Consider the partition
$$\MB=\begin{pmatrix} \MB_{11}& \MB_2 \\ \MB_3 & \MB_4 \end{pmatrix}.$$
Then $(\MB_4-\MB_3\MB_{11}^{-1}\MB_2)(a')$ is invertible since it is the Schur complement of $\MB_{11}(a')$, so $\MB_4-\MB_3\MB_{11}^{-1}\MB_2$ is invertible by the induction hypothesis. Therefore $\MB$ is invertible.
\end{proof}

Our definition of mp rational functions admits a convenient interpretation of a partial evaluation with respect to $X^{(1)}$ in terms of matrices over mp rational functions in the remaining variables as is shown in the following proposition.

\begin{prop}\label{p:matval}
Let $\rrb\in \mpr$ and $a\in \Mat_{d}(\kk)^{g_1}$ be such that
$(a,b)\in \dom_{d,n_2,\dots,n_G}\rrb$
for some $b\in \cM_{n_2,\dots,n_G}^{(g_2,\dots,g_G)}$. 
Then there exists $\SB\in \Mat_d(\prempr{2}{G})$ such that $\rrb(a,c)=\SB(c)$ for all $c\in\dom \SB$ such that $(a,c)\in \dom \rrb$.
\end{prop}

\begin{proof}
Let $r$ be a representative of $\rrb$ with $(a,b)\in \mdom[d,n_2,\dots,n_G]r$. Consider the $d\times d$ matrix-valued rational expression $s$ in $X^{(2)},\dots, X^{(G)}$ over $\kk$ which we get by replacing $X^{(1)}_j$ with $a_j$ in $r$ (cf.~\cite[Definition 2.1]{KVV1} or \cite[Section 2]{KVV2}). By the definition of the Kronecker product we have $r(a,c)=s(c)$ for all $c\in\mdom s$ such that $(a,c)\in\mdom r$. By induction on the inversion height and repetitive application of Proposition \ref{p:inv} we can see that $s$ can be represented as a $d\times d$ matrix $S$ whose entries are mp-nondegenerate rational expressions in $X^{(2)},\dots, X^{(G)}$ over $\kk$ (cf.~\cite[Remark 2.16]{KVV1}). This $S$ determines $\SB\in \Mat_d(\prempr{2}{G})$ with the desired property.
\end{proof}

\subsection{Intersections and centralizers}\label{subsec32}

Theorem \ref{t:div} implies that $\prempr{i_1}{i_k}$ naturally embeds into $\mpr$ for $1\le i_1<\dots <i_k\le G$ and $k\le G$. Here we establish intersection and commutation relations between these embeddings, which reflect the corresponding relations between subrings $\px{i}$ of $\txx$. While the results are not surprising, their  proofs are somewhat subtle since mp rational functions are defined as equivalence classes with respect to matrix evaluations.

\begin{lem}\label{l:indep}
Let $\rrb\in\mpr$ and assume that $\rrb$ is independent of $X^{(1)}$ on each level set, i.e.,
\begin{equation}\label{e:1st}
\rrb(a'^{(1)},b)=\rrb(a''^{(1)},b)
\end{equation}
for all $(a'^{(1)},b), (a''^{(1)},b)\in\mdom \rrb$ such that the sizes of matrices in $a'^{(1)}$ and $a''^{(1)}$ coincide. Then $\rrb\in\prempr{2}{G}$.
\end{lem}

\begin{proof}
Fix an arbitrary $\tilde{a}\in \mm{m_1}^{g_1}$ such that $(\tilde{a},\tilde{b})\in \dom_{m_1,m_2,\dots,m_G}\rrb$ and let
$$\Omega_{n_2,\dots,n_G}=\left\{b\in \cM^{\ug}_{n_2,\dots,n_G}\colon (\tilde{a},b)\in \dom_{m_1,n_2,\dots,n_G} \rrb\right\},
\qquad \Omega=\bigcup_{n_2,\dots,n_G} \Omega_{n_2,\dots,n_G}.$$
By \eqref{e:1st} and Proposition \ref{p:matval} we can find $\SB\in \Mat_{m_1}(\prempr{2}{G})$ such that $\rrb(\tilde{a},b)=\RS(b)$ for almost every $b\in\Omega$. Proposition \ref{p:basic}(2) then implies
\begin{align*}
\SB(b)
&=\rrb(\tilde{a},b)=\rrb(p\tilde{a}p^{-1},b) \\
&=(p\otimes I)\rrb(\tilde{a},b)(p^{-1}\otimes I)=(p\otimes I)\SB(b)(p^{-1}\otimes I)
\end{align*}
for almost every $b\in\Omega$ and every $p\in \GL_{m_1}(\kk)$. But this implies $\SB=p\SB p^{-1}$ for all $p\in \GL_{m_1}(\kk)$, so $\SB=\rsb I_{m_1}$ for some $\rsb\in\prempr{2}{G}$. Hence we have
\begin{equation}\label{e:2nd}
\rrb(\tilde{a},b)=I_{m_1}\otimes \rsb(b)
\end{equation}
for almost every $b\in\Omega$.

We claim that $\rrb=\rsb$, i.e.,
\begin{equation}\label{e:3rd}
\rrb(a,b)=I\otimes \rsb(b)
\end{equation}
holds for almost every $(a,b)\in \mdom\rrb$. Let
$$\widehat{\Omega}_{k_1,\dots,k_G}=\left\{(a,b)\in\dom_{k_1m_1,\dots,k_Gm_G}\rrb\colon b\in \Omega\right\},
\qquad \widehat{\Omega}=\bigcup_{k_1,\dots,k_G}\widehat{\Omega}_{k_1,\dots,k_G}.$$
The sets $\widehat{\Omega}_{k_1,\dots,k_G}$ are non-empty for all $k_1,\dots, k_G\in\N$ by Remark \ref{r:updown}. Moreover, $\widehat{\Omega}_{k_1,\dots,k_G}$ is Zariski-open in $\cM^{\ug}_{k_1m_1,\dots,k_Gm_G}$ for every choice of $k_i\in\N$. By Remark \ref{r:updown} and a density argument it is thus enough to prove that \eqref{e:3rd} holds on $\widehat{\Omega}$. But \eqref{e:3rd} holds on $\widehat{\Omega}_{1,k_2,\dots,k_G}$ by \eqref{e:2nd} and \eqref{e:1st}, and consequently holds on $\widehat{\Omega}_{k_1,\dots,k_G}$ by Proposition \ref{p:basic}(1) and \eqref{e:1st}.
\end{proof}

\begin{lem}\label{l:cap}
Let $1\le G_0\le G_1\le G$. Then
\begin{equation}\label{e:cap}
\prempr{1}{G_1}\cap \prempr{G_0}{G}=\prempr{G_0}{G_1}
\end{equation}
holds in $\mpr$.
\end{lem}

\begin{proof}
While the inclusion $\supseteq$ in \eqref{e:cap} is obvious, the inclusion $\subseteq$ holds by Lemma \ref{l:indep} and its variants for $X^{(i)}$, which hold by Remark \ref{r:shuffle}.
\end{proof}

Let $\cent(S)$ denote the centralizer of set $S$ in $\mpr$.

\begin{lem}\label{l:cent}
If $|X^{(1)}|>1$, then $\cent(\rx{1})=\prempr{2}{G}$.
\end{lem}

\begin{proof}
Obviously we have $\prempr{2}{G}\subseteq \cent(\rx{1})$. Conversely, assume 
$\rrb\in\cent(\rx{1})$. For $n_1\in\N$ and $b\in \cM^{\ug}_{n_2,\dots,n_G}$ let 
$n=n_2\cdots n_G$ and
$$\Omega_{b,n_1}=\left\{a\in \mm{n_1}^{g_1}\colon (a,b)\in \dom_{\seqn} \rrb\right\}.$$
If $\Omega_{b,n_1}\neq\emptyset$, then Proposition \ref{p:matval} and \eqref{e:tenspi} imply
$$\rrb(a,b)^{\mps}=K^t \RS_b(a)K$$
for some $\RS_b\in \Mat_n(\rx{1})$ and almost every $a\in \Omega_{b,n_1}$, where $K$ is the commutation matrix corresponding to the transposition of 1 and $G$. Since $\rrb X^{(1)}_j-X^{(1)}_j\rrb=0$ by assumption, we have
$$K^t \RS_b(a)K (a_j\otimes I_n)-(a_j\otimes I_n)K^t \RS_b(a)K=0,$$
which is equivalent to
$$\RS_b(a)(I_n\otimes a_j)-(I_n\otimes a_j)\RS_b(a)=0.$$
Therefore
$$\RS_b\begin{pmatrix}X^{(1)}_j&&\\&\ddots&\\&&X^{(1)}_j\end{pmatrix}-\begin{pmatrix}X^{(1)}_j&&\\&\ddots&\\&&X^{(1)}_j\end{pmatrix}\RS_b=0$$
holds. Since the center of $\rx{1}$ is $\kk$ by \cite[Corollary 7.9.7]{Co3}, we have $\RS_b\in\mm{n}$ and hence $\rrb(a',b)=\rrb(a'',b)$ for all $a',a''\in\Omega_{b,n_1}$. Therefore $\rrb\in\prempr{2}{G}$ by Lemma \ref{l:cap}.
\end{proof}

\begin{prop}\label{p:centralizer}
Let $0\le G_0\le G_1\le G$ and assume $|X^{(i)}|>1$ for $i\le G_0$ and $|X^{(i)}|=1$ for $G_0< i\le G_1$. Then
\begin{equation}\label{e:cent}
\cent(\prempr{1}{G_1})=\prempr{G_0+1}{G}.
\end{equation}
\end{prop}

\begin{proof}
Inclusion $\supseteq$ in \eqref{e:cent} is clear. On the other hand,
\begin{align*}
\cent(\prempr{1}{G_1})
&\subseteq \bigcap_{i=1}^{G_0}\cent(\rx{i}) \\
&=\bigcap_{i=1}^{G_0} \kk\plangle \cdots \mpsymbol X^{(i-1)} \mpsymbol X^{(i+1)} \mpsymbol \cdots \prangle \\
&=\prempr{G_0+1}{G}
\end{align*}
holds by Lemma \ref{l:cap}, Lemma \ref{l:cent} and its variants due to Remark \ref{r:shuffle}.
\end{proof}

\section{The universal property of \texorpdfstring{$\mpr$}{k(X(1) ... X(G))}}\label{sec4}

In this section we prove our main result, Theorem \ref{t:univ}: mp rational functions form the universal skew field of fractions of the multipartite free algebra, i.e., the tensor product of free algebras. This is achieved in Subsection \ref{subsec42} after we develop all the tools needed in Subsection \ref{subsec41}. Finally, in Subsection \ref{subsec43} we show that a rational expression vanishes on a multipartite variety of $\mm{n}$ for all $n\in\N$ if and only if it vanishes on the corresponding multipartite variety of every skew field.

\subsection{Rational expressions over a skew field}\label{subsec41}

The main aim of this subsection is to derive the tools needed for proving the universality of $\mpr$. However, some of the results are interesting in their own right. Let $D$ be an arbitrary skew field whose center contains $\kk$ and let $Z=\{Z_1,\dots,Z_g\}$ a set of freely noncommuting variables. The $\kk$-algebra $\pz=D\kotimes \kk\Langle Z\Rangle$ is called the {\bf free $D$-ring on $Z$}. By \cite[Corollary 2.5.2]{Co3}, $\pz$ is a fir and its universal skew field of fractions is denoted $\rz$. An alternative and explicit construction of this skew field via matrix evaluations is stated as Corollary \ref{c:rz}.

\begin{rem}\label{r:dt}
For later reference we recall some classical facts. The $\kk$-algebra $D[t]=D\kotimes \kk[t]$ is an Ore domain \cite[Theorem 1.2.9(iv) and Theorem 2.1.15]{MR} and thus has a classical ring of quotients $D(t)$, i.e., every element in $D(t)$ is of the form $p^{-1}q$ for some $p\in D[t]\setminus\{0\}$ and $q\in D[t]$. Also, since $\kk$ is infinite, a Vandermonde matrix argument \cite[Proposition 2.3.27]{Row} implies that $p\neq0$ if and only if $p(\alpha)\neq0$ holds for almost every $\alpha\in\kk$. Lastly, there is a valuation $v:D(t)\to\Z\cup\{\infty\}$ defined by $v(0)= \infty$ and $v(p^{-1}q)= \deg(p)-\deg(q)$ for $p,q\in D[t]\setminus\{0\}$; see e.g. \cite[Section 9.1]{Co1}.
\end{rem}

We start by proving some technical results. In general, the tensor product of two skew fields is not necessarily a domain \cite{RS}; however, we will show that the tensor product of $D$ with a generic division algebra $\ud{m}(z)$ (where $z$ is a tuple of generic $m\times m$ matrices) embeds into a skew field. At the heart of the next proof is a construction of a generalized cyclic division algebra (cf.~\cite[Section 1.4]{Jac}).

\begin{lem}\label{l:aux1}
Let $m\in \N$. Then there exists a cyclic algebra $\cA$ of degree $m$ whose center contains $\kk$ such that $D\kotimes \cA$ is a skew field.
\end{lem}

\begin{proof}
Let $E$ be a skew field whose center contains $\kk$ and consider $\tilde{E}=E(t_0,\dots,t_{m-1},t)$ for algebraically independent commutative symbols $t_0,\dots,t_{m-1},t$. Let $\sigma:\tilde{E}\to \tilde{E}$ be the automorphism determined by
$$\sigma|_E=\id_E,\quad \sigma t=t,\quad \sigma t_0=t_1,\quad \sigma t_1=t_2,\quad \dots, \quad \sigma t_{m-1}=t_0.$$
The ring $\tilde{E}[u;\sigma]$ is a principal right ideal domain by \cite[Theorem 1.2.9(ii)]{MR}. Consider 
the central element
$u^m-t\in \tilde{E}[u;\sigma]$.
We shall use the Eisenstein criterion 
\cite{GMR} for skew polynomial rings over division algebras
to show it is  irreducible. If $m\ge2$, let $v$ be the $t$-adic valuation on $\tilde{E}=E(t_0,\dots,t_{m-1})(t)$ as in Remark \ref{r:dt}. Then it is easy to verify that $v$ extends to a valuation
$$\hat{v}:\tilde{E}[u;\sigma]\to\Z\cup\{\infty\},\qquad
\hat{v}\left(\sum_{i=0}^n a_iu^i\right)=\min\{v(a_i)-i\colon 0\le i\le n\}.$$

One can now verify that this setting satisfies the conditions of \cite[Theorem 38]{GMR} (with left key polynomial $u$), hence $u^m-t$ is irreducible. Therefore the two-sided ideal $(u^m-t)\subset \tilde{E}[u;\sigma]$ is maximal as a one-sided ideal, so the quotient ring
$$\cA(E,m)\eqcolon\tilde{E}[u;\sigma]/(u^m-t)$$
is a skew field.

Our statement now follows since
$$\cA(D,m)=D\kotimes \cA(\kk,m)$$
and the cyclic algebra $\cA(\kk,m)$ is of degree $m$ (see e.g. the proof of \cite[Proposition 3.1.46]{Row}).
\end{proof}

\begin{lem}\label{l:aux2}
Let $\cA$ be a $\kk$-algebra such that $D\kotimes \cA$ is a domain and let $p_1,p_2\in \pz$ be such that $p_1(a)p_2(a)=0$ for every $a\in \cA^g$. Then $p_1(a)=0$ for all $a\in \cA^g$ or $p_2(a)=0$ for all $a\in \cA^g$.
\end{lem}

\begin{proof}
Assume this is not the case, i.e., there exist $b_1,b_2\in\cA^g$ such that $p_1(b_1)\neq0$ and $p_2(b_2)\neq0$. Now consider
$$q_i(t)=p_i((1-t)b_1+tb_2) \in (D\kotimes \cA)[t];$$
because $q_1(\alpha)q_2(\alpha)=0$ for every $\alpha\in\kk$ and $\kk$ is infinite, we have $q_1(t)q_2(t)=0$ by Remark \ref{r:dt}. Since $(D\kotimes \cA)[t]$ is a domain, we have $q_1(t)=0$ or $q_2(t)=0$. But $q_1(0)\neq0$ and $q_2(1)\neq0$, a contradiction.
\end{proof}

For a given $m\in\N$ denote
$$\kk[\xi]=\kk[\xi_{i\iimath\jjmath}:1\le i\le g, 1\le \iimath,\jjmath \le m]$$
and let $z_i=(\xi_{i\iimath\jjmath})_{\iimath\jjmath}$ be generic $m\times m$ matrices.

\begin{prop}\label{p:tensgen}
$D\kotimes \ud{m}(z)\subseteq \Mat_m(D(\xi))$ is a Noetherian domain. In particular, $D\kotimes \gm{m}(z)$ has a classical ring of quotients 
$\ud{m}(D;z)\subseteq \Mat_m(D(\xi))$. If $E\supseteq D$ is another skew field, then $\ud{m}(D;z)$ embeds into $\ud{m}(E;z)$ and the diagram
\begin{center}
	\begin{tikzpicture}[scale=1]
	\node (A) at (0,1.5) {$D$};
	\node (B) at (3,1.5) {$E$};
	\node (C) at (0,0) {$\ud{m}(D;z)$};
	\node (D) at (3,0) {$\ud{m}(E;z)$};
	\path[->]
	(A) edge (B)
	(A) edge (C)
	(B) edge (D)
	(C) edge (D);
	\end{tikzpicture}
\end{center}
commutes.
\end{prop}

\begin{proof}
Let $K$ be the center of $\ud{m}(z)$. Then $D\kotimes K$ is a Noetherian $\kk$-algebra by \cite[Theorem 3]{RSW} because $K$ is an intermediate field extension of $\kk(\xi)/\kk$ and therefore a finitely generated field extension of $\kk$. Since $\ud{m}(z)$ is finite-dimensional over $K$, $D\kotimes \ud{m}(z)$ is a finitely generated module over $D\kotimes K$, so $D\kotimes \ud{m}(z)$ is a Noetherian $\kk$-algebra by \cite[Lemma 1.1.3]{MR}.

Next we need to prove that $D\kotimes \ud{m}(z)$ is a domain. Since $\ud{m}(z)$ is the ring of central quotients of $\gm{m}(z)$, it is enough to prove that $D\kotimes \gm{m}(z)$ is a domain. Let $p_1,p_2\in \pz$ and assume that $\hat{p}_1\hat{p}_2=0$ holds for their canonical images $\hat{p}_1,\hat{p}_2\in D\kotimes \gm{m}(z)$. Let $\cA$ be a cyclic algebra of degree $m$ as in Lemma \ref{l:aux1}. For every $a\in\cA^g$ we have a homomorphism $\gm{m}(z)\to \cA$ given by $z\mapsto a$; therefore $p_1(a)p_2(a)=0$ for every $a\in\cA^g$. Lemma \ref{l:aux2} then without loss of generality implies $p_1(a)=0$ for all $a\in\cA^g$. Since
$$(D\kotimes \cA)\Cotimes \bC \cong D\kotimes \Mat_m(\bC),$$
where $\bC$ is the algebraic closure of the center $C$ of $\cA$, we have $\hat{p}_1=0$.

The second part of the proposition follows from the first part because every Noetherian domain is an Ore domain and $\ud{m}(z)$ is the ring of central quotients of $\gm{m}(z)$. Finally, the third part follows from the observation that the embedding $D\hookrightarrow E$ extends to the natural embedding $D\kotimes \gm{m}(z)\hookrightarrow E\kotimes \gm{m}(z)$.
\end{proof}

\def\pzz{\widehat{D}\Langle Z \Rangle}
\def\rzz{\widehat{D}\plangle Z \prangle}
\def\ppz{\widehat{D}\Langle\!\Langle Z \Rangle\!\!\Rangle}

Similarly to $\cR_{\kk}(Z)$, let $\rez$ be the set of all formal rational expressions built from $D$ and $Z$. A priori, these expressions are just combinations of symbols without relations; to study $\rz$ with them, where variables $Z$ commute with elements of $D$., we will consider their evaluations that model these commutativity relations (unlike e.g. \cite[Section 8.2]{Row} where evaluations without commutativity restrictions are considered).
As before, we have a notion of the inversion height of an expression $r\in\rez$. If $a\in\mm{m}^g$, then every $p\in\pz$ yields $$p(a)\in D\kotimes\mm{m}=\Mat_m(D).$$
This evaluation can be extended to rational expressions in a natural way. Each expression of height 0 yields an element of $\pz$ which can be evaluated as above, and evaluations of expressions of higher degree are then defined recursively. If $r\in\rez$ is defined at $a\in\mm{m}^g$, then we have $r(a)\in \Mat_m(D)$; if $r(a)$ is moreover invertible in $\Mat_m(D)$, then $r^{-1}\in\rez$ is defined at $a$. Note that $r$ is either not defined on $\mm{m}^g$ or defined in almost every point in $\mm{m}^g$ by Remark \ref{r:dt}. We call $r$ {\bf non-degenerate} if $r$ is defined at some $a\in\mm{m}^g$.

\begin{prop}\label{p:idD}
Let $r\in\rez$ and assume it vanishes on $\bigcup_m\mm{m}^g$. Then $r$ represents $0$ in $\rz$.
\end{prop}

\begin{proof}
By assumption, $r$ is defined at some point in $\mm{m}^g$. Hence it is also defined at the tuple of generic matrices $z$ from $\gm{m}(z)\subseteq \ud{m}(D;z)=\widehat{D}$. Since the latter is a skew field, $r$ indeed represents an element of $\rz$. Let $\hat{r}(Z)=r(Z+z)$ be a rational expression over $\widehat{D}$; since it is defined at 0, it represents an element of $\rzz$.

Observe that $r$ represents 0 if $\hat{r}$ represents 0. Indeed: consider the homomorphism $\phi:\gm{m}(z)\kotimes \pz\to \pz$ determined by $\phi(z)=0$ and $\phi|_{\pz}=\id_{\pz}$. Since $\rzz$ is a skew field of fractions of $\gm{m}(z)\kotimes \pz$, there exists a subring $\gm{m}(z)\kotimes \pz\subseteq L\subset\rzz$ maximal with the property that $\phi$ extends to a homomorphism $\varphi:L\to \rz$. By induction on the inversion height of $r$ we see that $\hat{r}\in L$ and $\varphi(\hat{r})=r$, so $r\neq0$ implies $\hat{r}\neq0$.

If $r$ is defined at $a\otimes I_m+I_n\otimes b$ for $a\in\mm{n}^g$ and $b\in\mm{m}^g$, then the definition of evaluation of rational expressions over skew fields on tuples of matrices implies
\begin{equation}\label{e:e1}
\hat{r}(a)=r(a\otimes I_m+I_n\otimes z)=0,
\end{equation}
where the second equality holds by assumption.

Let $\ppz$ be the $(Z)$-adic completion of $\pzz$. Since $\ppz$ is a semifir by \cite[Theorem 5.4.5]{Co1} and the embedding $\pzz\to \ppz$ is honest by \cite[Proposition 6.2.2]{Co1}, it extends to an embedding of $\rzz$ into the universal skew field of $\ppz$ by \cite[Theorem 4.5.10]{Co1}. Since $\hat{r}$ is defined at 0, it can be expanded into a series $S\in\ppz$, and $\hat{r}$ represents 0 if $S=0$ by what we just observed.

Let $S=\sum_{w\in \,\Langle Z\Rangle} c_ww$ for $c_w\in\widehat{D}$. If $\hat{r}$ is defined at $a\in\mm{n}^g$, then
$$\frac{\der}{\der t^h} \hat{r}\big(ta\big)\Big|_{t=0}=h!\sum_{|w|=h} c_ww(a)$$
for every $h\in\N\cup\{0\}$. Let $p_h=\sum_{|w|=h} c_ww\in \pzz$. By \eqref{e:e1} and a density argument we see that $p_h(a)=0$ for every $a\in\mm{n}^g$ and $n\in\N$. As in the proof of \cite[Lemma 1.4.3]{Row}, we can use a ``staircase'' of standard matrix units to show that $p_h=0$. Hence $S=0$ and thus $r$ represents 0 in $\rz$.
\end{proof}

The above results yield the following characterization of $\rz$ in terms of evaluations on matrices over $\kk$. We introduce an equivalence relation on the set of all non-degenerate rational expressions in $\rez$ as in \eqref{e:rel}. The set of equivalence classes $\cE$ is a ring under natural operations.

\begin{cor}\label{c:rz}
Let $D\supseteq \kk$ be a skew field.	
\begin{enumerate}[(a)]
\item Let $r\in\rez$. Then $r$ represents an element in $\rz$ if and only if $r$ is non-degenerate.
\item $\cE$ is isomorphic to $\rz$.
\end{enumerate}	
\end{cor}

\begin{proof}
(a) By induction on the inversion height of $r$ it suffices to prove the following: if $s\in\rez$ represents an element in $\rz$ and $s$ is non-degenerate, then $s$ represents $0$ in $\rz$ if and only if $s$ vanishes on $\bigcup_m \mm{m}^g$. The implication $(\Leftarrow)$ holds by Proposition \ref{p:idD}. Conversely, to prove $(\Rightarrow)$ let $s$ be defined on $\mm{m}^g$. Then $s(z)\in\ud{m}(D;z)$. Since $\ud{m}(D;z)$ is a skew field by Proposition \ref{p:tensgen} and $s$ represents $0$ in $\rz$, we have $s(z)=0$ by the universality of $\rz$. Therefore $s$ vanishes on $\mm{m}^g$.

(b) Consider the map $\phi:\rz\to\cE$ sending the element in $\rz$ represented by $r\in\rez$ to the equivalence class of $r$. By the proof of (a), this is a well-defined injective homomorphism. Since $\phi|_{\pz}=\id_{\pz}$ and $\rz,\cE$ are skew fields generated by $\pz$, $\phi$ is surjective and therefore an isomorphism.
\end{proof}

Let $R$ be an arbitrary ring. A matrix $M\in \Mat_n(R)$ is {\bf full over $R$} if it is not a product of smaller rectangular matrices over $R$. If $R$ is a fir, then every full matrix over $R$ is invertible in the universal skew field of fractions of $R$ by \cite[Corollary 4.5.9]{Co1}.

\begin{cor}\label{c:rzmat}
Let $M$ be a matrix over $\pz$. If $M$ is full over $\pz$, then $M(a)$ is invertible for some $a\in\mm{m}^g$.
\end{cor}

\begin{proof}
Since $\pz$ is a fir, $M$ is invertible over $\rz$. Let $N$ be a matrix of rational expressions that are representatives of the entries in the inverse of $M$ in $\rz$. If $a$ belongs to the intersection of domains of entries in $N$, we have $M(a)N(a)=I$. Such an $a$ exists by Corollary \ref{c:rz}(a).
\end{proof}

\subsection{Main theorem}\label{subsec42}

We are finally in a position to prove the universal property of $\mpr$.

\begin{thm}\label{t:univ}
The skew field $\mpr$ is the universal skew field of fractions of $\txx$.
\end{thm}

\def\MM{\widetilde{M}}

\begin{proof}
The assertion is proved by induction on $G$. The basis case $G=1$ is presented in \cite[Proposition 2.2]{KVV2} as a consequence of Amitsur's theorem on rational identities \cite[Theorem 16]{Ami}. Therefore let $G\ge 2$; we will use the characterization of universality from Remark \ref{r:another}.

Let $M$ be a $d\times d$ matrix over $\txx$ and assume there exists a skew field $D$ and a homomorphism $\varphi:\txx\to D$ such that $\varphi(M)$ is invertible over $D$. Set $b'=\varphi(X^{(1)})$ and $b=(\varphi(X^{(2)}),\dots,\varphi(X^{(G)}))$. Then $\widetilde{M}=M(X^{(1)},b)$ is a $d\times d$ matrix over $D\Langle X^{(1)}\Rangle$. Since $\widetilde{M}(b')$ is invertible, $\widetilde{M}$ is full over $D\Langle X^{(1)}\Rangle$ and Corollary \ref{c:rzmat} implies that $\widetilde{M}(a')$ is invertible for some $a'\in\mm{n_1}^{g_1}$.

Let $N$ be the $dn_1\times dn_1$ matrix over $\pretxx{2}{G}$ obtained from $M$ by substituting $X^{(1)}$ with $a'$. In particular, we have
$$N(b)=M(a',b)=\widetilde{M}(a').$$
By the induction hypothesis, $\prempr{2}{G}$ is the universal skew field of fractions of $\pretxx{2}{G}$. Since $N(b)$ admits an invertible evaluation $N(b)$ over $D$, $N$ is invertible over $\prempr{2}{G}$ by the universality. Therefore $N(a)$ is invertible for some $a\in\cM^{\ug}_{n_2,\dots,n_G}$ by Proposition \ref{p:inv}. Since $M(a',a)=N(a)$ is invertible, $M$ is invertible in $\mpr$ by Proposition \ref{p:inv}.
\end{proof}

\begin{cor}\label{c:univ}
Every finite tensor product of free algebras has a universal skew field of fractions.
\end{cor}

\begin{proof}
For $1\le i\le G$ let $S^{(i)}$ be an index set and consider $\cA=\kk\Langle S^{(1)} \mpsymbol\cdots\mpsymbol S^{(G)}\Rangle$. Let $\cF(S^{(i)})$ be the set of all finite subsets of $S^{(i)}$ and endow $\cS=\prod_i\cF(S^{(i)})$ with the partial order
$$X \preceq X' \iff \forall i\colon X^{(i)}\subseteq X'^{(i)}$$
for $X,X'\in \cS$. By Theorem \ref{t:univ}, $\txx$ has the universal skew field of fractions $\cU(X)$ for every $X\in\cS$. Moreover, if $X\preceq X'$, then $\cU(X)$ naturally embeds into $\cU(X')$ by Remark \ref{r:shuffle}. Since $(\cS,\preceq)$ is a lattice and $\{\cU(X)\colon X\in\cS\}$ together with aforementioned natural embeddings is a directed system, there exists the direct limit of skew fields
$$\cU=\varinjlim \cU(X).$$
It is easy to see that $\cU$ is a skew field of fractions of $\cA$. Let $M$ be a matrix over $\cA$ and assume that the image of $M$ is invertible under some homomorphism to a skew field. By looking at the entries of $M$ we conclude that $M$ is a matrix over $\txx$ for some $X\in\cS$. But then $M$ is invertible as a matrix over $\cU(X)$ and thus also as a matrix over $\cU$. Hence $\cU$ is the universal skew field of fractions of $\cA$.
\end{proof}

A consequence of Corollary \ref{c:univ} is also the following statement in the terms of group algebras. For related results we refer to \cite{Le,Pa}.

\begin{cor}
The group $\kk$-algebra of a finite direct product of free groups admits a universal skew field of fractions.
\end{cor}

\subsection{Vanishing on multipartite varieties}\label{subsec43}

For any ring $R$ let
$$\mpvar{R}=\left\{(a^{(1)},\dots,a^{(G)})\in R^{g_1}\times\cdots\times R^{g_G}\colon 
\left[a^{(i_1)}_{j_1},a^{(i_2)}_{j_2}\right]=0 \ \forall i_1\neq i_2 \ \forall j_1,j_2\right\}$$
be the {\bf multipartite variety} associated with $R$.

\begin{prop}\label{p:comm}
Let $r\in\re$.
\begin{enumerate}[\rm(a)]
\item If $r$ nc-vanishes on $\mpvar{\mm{n_1\cdots n_G}}$, then it mp-vanishes or is mp-undefined 
on $\cM^{\ug}_{\seqn}$.
\item If $r$ mp-vanishes on $\cM^{\ug}_{n,\dots,n}$, then it nc-vanishes or is nc-undefined 
on $\mpvar{\mm{n}}$.
\end{enumerate}
\end{prop}

\begin{proof}
(a) This is clear since $\tau(\cM^{\ug}_{\seqn})\subseteq \mpvar{\mm{n_1\cdots n_G}}$.

(b) Let $z$ be the $\sum_ig_i$-tuple of $n^G\times n^G$ generic matrices. Since $r(z)\in \ud{n^G}(z)$, there exist $p,q\in \ppx$ such that
$$r(z)=p(z)q(z)^{-1}.$$
By assumption and Zariski density we see that $p^{\mps}(a)=0$ for all $a\in\cM^{\ug}_{n,\dots,n}$. Now let $b\in \mpvar{\mm{n}}$ be arbitrary and denote by $\cB\subseteq \mm{n^G}$ the unital $\kk$-subalgebra generated by $\tau(b)$. We observe that the restriction of the $\kk$-linear map
$$\ell:\mm{n^G}\cong \mm{n}^{\otimes G}\to \mm{n},\quad c_1\otimes\cdots\otimes c_G
\mapsto c_1\cdots c_G$$
to the linear map $\ell|_{\cB}:\cB\to \mm{n}$ is actually a homomorphism of $\kk$-algebras. But then
$$p(b)=\ell|_{\cB}(p(\tau(b)))=\ell|_{\cB}(0)=0.$$
Therefore $r(b)=0$ if $r$ is defined at $b$.
\end{proof}

Proposition \ref{p:comm} implies that there is a well-defined notion of nc-evaluation of a mp rational function on $\mpvar{\mm{n}}$: if a representative $r$ of a mp rational function $\rrb$ is nc-defined on $\mpvar{\mm{n}}$, then we can set $\rrb(a)^{\operatorname{nc}}=r(a)$ for all $a\in\dom_n r\cap\mpvar{\mm{n}}$ and $r(a)$ is independent of the choice of the representative $r$.

Noncommutative rational identities are reasonably well-understood due to the work of Amitsur and Bergman \cite{Ami,Be1,Be2}. The following result is a weak multipartite version of Amitsur's theorem on rational identities.

\begin{thm}\label{t:sci}
Let $r\in \re$. The following are equivalent:
\begin{enumerate}[\rm(i)]
\item $r$ is mp-defined and nonzero on $\cM^{\ug}$;
\item $r$ is nc-defined and nonzero on $\mpvar{\mm{n}}$ for some $n\in\N$;
\item $r$ is nc-defined and nonzero on $\mpvar{D}$ for some skew field $D$.
\end{enumerate}
\end{thm}

\begin{proof}
The equivalence of (i) and (ii) follows by Proposition \ref{p:comm}. Since $\udn$ are skew fields, (i) implies (iii).

(iii)$\Rightarrow$(i). We prove the claim by induction on the inverse height of $r$. Let $r$ be defined and nonzero at $a=(a^{(1)},\dots,a^{(g)})\in\mpvar{D}$ and consider the homomorphism $\varphi:\txx\to D$ defined by $x^{(i)}\mapsto a^{(i)}$. The basis of induction now holds because $\rrb\in \txx\setminus\{0\}$ if there are no inverses appearing in $r$. By Theorem \ref{t:univ} and the universal property, there exists a local homomorphism $\lambda:\mpr\dashrightarrow D$ that extends $\varphi$. Since all sub-expressions of $r$ are defined and nonzero at $a$, the induction hypothesis implies $\rrb_0\neq 0$ for all $r_0\in \re$ such that $r_0^{-1}$ is a sub-expression in $r$. Therefore $r$ is a mp-nondegenerate expression and $\rrb$ lies in the domain of $\lambda$. Since $\lambda(\rrb)=r(a)\neq0$, we have $\rrb\neq0$.
\end{proof}

To obtain a full version of Amitsur's theorem, one would need to show that for an arbitrary fixed infinite dimensional skew field $E$ whose center contains $\kk$, (i)--(iii) from Theorem \ref{t:sci} are equivalent to
\begin{enumerate}[\rm(iv)]
	\item $r$ is nc-defined and nonzero on $\mpvar{E}$.
\end{enumerate}

\section{Free noncommutative function theory perspective}\label{sec5}

In this section we explain how mp rational functions fit into the wider frame of free noncommutative function theory. They are essentially higher order nc functions in the sense of \cite{KVV3}. We also introduce the difference-differential operators for mp rational functions in Subsection \ref{subsec52}, and briefly discuss linearization or realization in Subsection \ref{subsec53}.

\subsection{Higher order noncommutative rational functions}\label{subsec51}

We now put mp rational functions into the setting of free function theory. Proposition \ref{p:basic} implies that elements of $\mpr$ are essentially nc functions of order $G-1$ of \cite[Section 3.1]{KVV3}; some care has to be taken because matrices $a\otimes (b\oplus c)$ and $(a\otimes b)\oplus (a\otimes c)$ are different in general, but always unitarily similar. As before, let
$$\cM^{\ug}_{\seqn}=\prod_{i=1}^G \mm{n_i}^{g_i},\qquad \cM^{\ug}=\bigcup_{\seqn} \cM^{\ug}_{\seqn}.$$
In addition, set
$$\cN_{\seqn}=\bigotimes_{i=1}^G \mm{n_i}=\mm{n_1\cdots n_G},\qquad
\cN=\bigcup_{\seqn} \cN_{\seqn}.$$
Each $\rrb\in\mpr$ yields a partially defined map $\rrb:\cM^{\ug} \dashrightarrow \cN$, which satisfies the following properties by Proposition \ref{p:basic} and \eqref{e:tenspi}:
\begin{enumerate}\label{ncfun}
\item $\rrb$ respects direct sums in the first factor and $\rrb$ respects direct sums in other factors up to conjugation by a permutation matrix;
\item $\rrb$ respects similarities in every factor.
\end{enumerate}
By \cite[Section 3.1]{KVV3} (and especially \cite[Remark 3.5]{KVV3}, which is relevant for our tensor product setting) and a slight loosening of the definition, we can thus say that $\rrb$ is a nc rational function of order $G-1$. The results on higher order nc functions from \cite[Chapter 3]{KVV3} still hold for mp rational functions if we replace equalities with equivalences up to a canonical shuffle, i.e., conjugation by a matrix built from appropriate commutation matrices.

\subsection{Difference-differential operators}\label{subsec52}

Next we describe partial difference-differential operators for mp rational functions (cf.~\cite[Section 3.5]{KVV3}). Since we defined them as equivalence classes of rational expressions, we can proceed as in \cite[Definition 4.4]{KVV2}. Let $1\le i\le G$ and $X'^{(i)}=\{X_1'^{(i)},\dots, X_{g_i}'^{(i)}\}$. For $1\le j\le g_i$ we define a map
\begin{equation}\label{e:delta}
\Delta^{(i)}_j\colon \re\to \cR_{\kk}(X'^{(i)}\cupn X)
\end{equation}
recursively by the following rules ($\alpha,\beta\in\kk$ and $r,s\in\re$):
\begin{subequations}\label{e:rules}
\begin{align}
\Delta^{(i)}_j(\alpha)&=0, \\
\Delta^{(i)}_j(X_{\jjmath}^{(\iimath)})&=\delta_{\iimath=i,\jjmath=j}, \\
\Delta^{(i)}_j(r+s)&=\Delta^{(i)}_j(r)+\Delta^{(i)}_j(s);\\
\Delta^{(i)}_j(r s)&=r(\dots,X'^{(i)},\dots) \Delta^{(i)}_j(s)+\Delta^{(i)}_j(r) s(\dots,X^{(i)},\dots), \\
\Delta^{(i)}_j(r^{-1})&=-r(\dots,X'^{(i)},\dots)^{-1} \Delta^{(i)}_j(r) r(\dots,X^{(i)},\dots)^{-1}.
\end{align}
\end{subequations}

By definition it is clear that $\Delta^{(i)}_j$ maps mp-nondegenerate expressions to mp-non\-degenerate expressions. For $v\in\kk^{g_i}$ denote $v\cdot \Delta^{(i)}(r)\eqcolon \sum_j v_j \Delta^{(i)}_j(r)$. We have the following analog of \cite[Theorem 4.8]{KVV3}.

\begin{prop}\label{p:fund}
For $r\in\re$ assume
$$(a^{(1)},a^{(2)},\dots,a^{(G)}),(a'^{(1)},a^{(2)},\dots,a^{(G)})\in \mdom r.$$
Let $v\in\kk^{g_1}$ be arbitrary and denote $\bar{v}=(v_1 I\otimes I,\dots, v_{g_1} I\otimes I)$. Then
\begin{equation}\label{e:fund}
r\left(
\begin{pmatrix}
a'^{(1)}\otimes I & \bar{v}\\
0 & I\otimes a^{(1)}
\end{pmatrix},\dots\right)^{\mps}=
\begin{pmatrix}
r(a'^{(1)}\otimes I,\dots)^{\mps} & (v\cdot \Delta^{(1)}(r))(a'^{(1)},a^{(1)},\dots)^{\mps}\\
0 & r(I\otimes a^{(1)},\dots)^{\mps}
\end{pmatrix},
\end{equation}
where the identity matrices on the left side of each Kronecker product have the same size as the components of $a'^{(1)}$ and the identity matrices on the right side of the Kronecker products have the same size as the components of $a^{(1)}$.
\end{prop}

\begin{proof}
The formula \eqref{e:fund} is proved by induction on the construction of $\re$ using \eqref{e:rules}. We leave these routine computations to the reader.
\end{proof}

If we consider $\kk \plangle X^{(1)} \mpsymbol\cdots\mpsymbol X'^{(i)} \mpsymbol X^{(i)} \mpsymbol\cdots\mpsymbol X^{(G)} \prangle$ as a $\mpr$-bi\-module with $\kk$-linear actions
$$\rrb\cdot \rsb\eqcolon r(\dots,X'^{(i)},\dots)\rsb \qquad\text{and}\qquad \rsb\cdot \rrb\eqcolon \rsb\rrb$$
for $\rrb\in\mpr$ and $\rsb\in \kk \plangle X^{(1)} \mpsymbol\cdots\mpsymbol X'^{(i)} \mpsymbol X^{(i)} \mpsymbol\cdots\mpsymbol X^{(G)} \prangle$, then we obtain the following corollary.

\begin{cor}\label{c:der}
The map $\Delta^{(i)}_j$ induces a $\kk$-linear derivation
$$\Delta^{(i)}_j\colon\mpr\to \kk \plangle X^{(1)} \mpsymbol\cdots\mpsymbol X'^{(i)} \mpsymbol X^{(i)} \mpsymbol\cdots\mpsymbol X^{(G)} \prangle.$$
\end{cor}

\begin{proof}

It is enough to show that $\Delta^{(i)}_j$ preserves the equivalence classes with respect to mp-evaluations; linearity and Leibniz's rule then follow from \eqref{e:rules}. Let us thus assume that mp-evaluations of two rational expressions $r_1$ and $r_2$ coincide.

First consider the case $i=1$. Plugging the standard basis of $\kk^{g_1}$ for $v$ in the formula \eqref{e:fund} of Proposition \ref{p:fund} implies that mp-evaluations of $\Delta^{(1)}_j(r_1)$ and $\Delta^{(1)}_j(r_2)$ coincide. For $i>1$ we recall that
$$\kk\plangle X^{(1)} \mpsymbol\cdots\mpsymbol X^{(i)} \mpsymbol\cdots\mpsymbol X^{(G)} \prangle\cong 
\kk\plangle X^{(i)} \mpsymbol\cdots\mpsymbol X^{(1)} \mpsymbol\cdots\mpsymbol X^{(G)}\prangle$$
and this isomorphism corresponds to the conjugation on the evaluations by the appropriate commutation matrix. By applying it to the formula \eqref{e:fund} we can again conclude that mp-evaluations of $\Delta^{(i)}_j(r_1)$ and $\Delta^{(i)}_j(r_2)$ coincide.
\end{proof}

\begin{rem}
The partial difference-differential operators behave nicely with respect to various specializations between mp rational functions. For example, homomorphisms
\begin{align*}
\kk\Langle X^{(1)}\cupn X^{(2)} \mpsymbol\cdots\mpsymbol X^{(G)}\Rangle
&\to \kk\Langle X^{(1)} \mpsymbol\cdots\mpsymbol X^{(G)}\Rangle,\\
\kk\Langle X'^{(1)}\cupn X'^{(2)} \mpsymbol X^{(1)}\cupn X^{(2)} \mpsymbol\cdots\mpsymbol X^{(G)}\Rangle
&\to \kk\Langle X'^{(1)}\mpsymbol X^{(1)} \mpsymbol\cdots\mpsymbol X^{(G)}\Rangle,
\end{align*}
where the latter is given by $X'^{(2)}_j\mapsto X^{(2)}_j$, induce local homomorphisms
\begin{align*}
\kk\plangle X^{(1)}\cupn X^{(2)} \mpsymbol\cdots\mpsymbol X^{(G)}\prangle &\dashrightarrow \mpr ,\\
\kk\plangle X'^{(1)}\cupn X'^{(2)} \mpsymbol X^{(1)}\cupn X^{(2)} \mpsymbol\cdots\mpsymbol X^{(G)}\prangle
&\dashrightarrow \kk\plangle X'^{(1)}\mpsymbol X^{(1)} \mpsymbol\cdots\mpsymbol X^{(G)}\prangle
\end{align*}
by the universal property of these skew fields. According to the defining rules \eqref{e:rules} of partial difference-differential operators it is easy to verify that the diagram
\begin{center}
	\begin{tikzpicture}[scale=1]
	\node (A) at (0,1.5) {$\kk\plangle X^{(1)}\cupn X^{(2)} \mpsymboltikz\cdots\mpsymboltikz X^{(G)}\prangle$};
	\node (B) at (8,1.5) {$\kk\plangle X^{(1)} \mpsymboltikz\cdots\mpsymboltikz X^{(G)}\prangle$};
	\node (C) at (0,0) {$\kk\plangle X'^{(1)}\cupn X'^{(2)} \mpsymboltikz X^{(1)}\cupn X^{(2)} \mpsymboltikz\cdots\mpsymboltikz X^{(G)}\prangle$};
	\node (D) at (8,0) {$\kk\plangle X'^{(1)}\mpsymboltikz X^{(1)} \mpsymboltikz\cdots\mpsymboltikz X^{(G)}\prangle$};
	\path[->,dashed]
	(A) edge (B)
	(C) edge (D);
	\path[->]
	(A) edge node[right]{$\Delta^{(1)}_j$} (C)
	(B) edge node[right]{$\Delta^{(1)}_j$} (D);
	\end{tikzpicture}
\end{center}
commutes.
\end{rem}

\subsection{Realizations}\label{subsec53}

Lastly, we address the realization aspect of mp rational functions. We refer the reader to \cite{BR,BGM,KVV2,HMV,Vol} for the classical realization theory of nc formal power series and nc rational functions.

Let $Z=\{Z_1,\dots,Z_g\}$. If $r\in\rezz$ is a nondegenerate rational expression and $p\in\dom_m r$, then $r$ has a {\bf realization about $p$} as in \cite[Subsection 5.1]{Vol}: there exist $n,\rho\in \N$ and
$$\cc\in \mm{m}^{1\times n},\quad\bb\in \mm{m}^{n\times 1}, \quad
C_{ij},B_{ij}\in\mm{m}^{n\times n}$$
for $1\le i\le g$ and $1\le j\le \rho$ such that
\begin{equation}\label{e:real}
r=\cc\left(I_n-\sum_{i=1}^g\sum_{j=1}^{\rho} C_{ij}(Z_i-p_i) B_{ij}\right)^{-1}\bb
\end{equation}
holds on $\mm{sm}^g$ for $s\in\N$ wherever both sides are defined. Here the evaluation of the right-hand side of \eqref{e:real} at a point $a\in \mm{sm}^g$ is defined as
\begin{equation}\label{e:realeval}
\cc^{\iota}\left(I_{nsm}-\sum_{i=1}^g\sum_{j=1}^{\rho} C_{ij}^{\iota}(I_n\otimes(a_i-p_i^{\iota}))B_{ij}^{\iota}\right)^{-1}\bb^{\iota},
\end{equation}
where $\iota=I_s\otimes\id:\mm{m}\to \mm{sm}$ is applied entry-wise to $\cc$, $C_{ij}$, $B_{ij}$ and $\bb$. To shorten the notation we introduce the linear pencil
$$L(Z)=\sum_{i=1}^g\sum_{j=1}^{\rho} C_{ij} Z_i B_{ij}.$$
The realization \eqref{e:real} of $r$ is denoted $(\cc,L,\bb;p)$ and we say that $n$ is its dimension. The union of the sets
$$\dom_{sm}(\cc,L,\bb;p)=\{a\in\mm{sm}^g\colon \det(I-L(a-p))\neq0\}$$
over all $s\in\N$ is called the {\bf domain} of $(\cc,L,\bb;p)$. We refer to \cite[Subsection 3.3]{Vol} for the construction of a realization of $r$ about $a$ and to \cite[Section 4]{Vol} for obtaining realizations that are minimal (with respect to their dimension) among all the realizations of $r$ about $p$. We also recall the following two facts, which are relevant in our setting.

\begin{enumerate}[(I)]
\item \ {\kern-.333em}\cite[Theorem 5.5 and Corollary 5.9]{Vol}: if $\dom_m r\neq\emptyset$, then for almost every $p\in\dom_m r$ there exists a minimal realization $(\cc,L,\bb;p)$ of $r$ such that $\dom_{sm}(\cc,L,\bb;p)=\dom r(z)$ for every $s\in\N$, where $z$ is a $g$-tuple of $sm\times sm$ generic matrices.
\item \ {\kern-.333em}\cite[Theorem 5.10]{Vol}: the dimension of a minimal realization of $r$ about $p$ is independent of $p\in\dom r$ and is therefore an invariant of $r$.
\end{enumerate}

Let now $r\in\re$ be a mp-nondegenerate expression. Since $r^{\mps}(a)=r(\tau(a))$ for $a\in\mdom r$, we can use realizations of $r$ to compute its mp-evaluations. More precisely, 
if $(\cc,L,\bb;p)$ is a realization of $r$,
then
$$\rrb(a)=r(a)^{\mps}=\cc\big(I-L(\tau(a)-p)\big)^{-1}\bb$$
for all $a\in\mdom[\seqn] r\cap \dom_{n_1\cdots n_G}(\cc,L,\bb;p)$ and $\seqn\in\N$ such that $m$ divides $n_1\cdots n_G$. Hence we can interpret $(\cc,L,\bb;p)$ as a realization of the mp rational function $\rrb$. Moreover, if $(\cc,L,\bb;p)$ is a realization as in (I), then
$$\tau\left(\dom_{\seqn} \rrb\right) \subseteq \dom_{n_1\cdots n_G}(\cc,L,\bb;p).$$
Indeed, if $r'$ is a representative of $\rrb$ and $\mdom[\seqn] r'\neq\emptyset$, then
$$\tau\left(\mdom[\seqn] r'\right) \subseteq \tau\left(\dom r'(x)\right) = \tau\left(\dom r(x)\right) 
\subseteq \dom r(z) = \dom_{n_1\cdots n_G}(\cc,L,\bb;p),$$
where $z$ is a $\sum_i g_i$-tuple of generic matrices from $\gm{n_1\cdots n_G}(z)$ and $x$ is a $\sum_i g_i$-tuple of multipartite generic matrices from $\gm{\seqn}(x)$. Therefore we can use $(\cc,L,\bb;p)$ to compute every evaluation of $\rrb$.

For $\rrb\in\mpr$ and $p\in\dom \rrb$ let $d_p(\rrb)$ denote the minimum of dimensions of realizations about $p$ representing $\rrb$. In contrast with the results on realizations of nc rational functions \cite{HMV,Vol}, one cannot expect any uniqueness of minimal realizations for mp rational functions. However, using (II) it is not hard to see that $d_p(\rrb)=d(\rrb)$ is actually independent of $p$ and thus an invariant of $\rrb$, which measures the complexity of $\rrb$. This uniformity suggests that minimal realizations of $\rrb$ about points in $\dom\rrb$ can be considered as normal forms of $\rrb$. They compensate for the lack of a canonical form for elements in $\mpr$.

\begin{appendices}

\section{Bi-free rational functions}\label{appA}

In this appendix we briefly touch upon a variant of mp free variables which is ubiquitous in free probability (cf.~\cite[Subsection 1.5 and Definition 2.6]{Voi1}), namely bi-free variables. We shall also present a matrix model for a natural skew field of fractions of the algebra of bi-free variables.

For $g\in\N$ let
$$\bfa=\kk\Langle X_1,\dots,X_g,Y_1,\dots, Y_g\Rangle  \big/ \left([X_i,Y_j]\colon i\neq j\right)$$
be the {\bf bi-free algebra} (on $2g$ variables). Moreover, let
$$\cA = \kk\Langle \,\{X'_1,\dots,X'_g\} \mpsymbol \{X''_1,Y''_1\} \mpsymbol\cdots\mpsymbol \{X''_g,Y''_g\}
\mpsymbol \{Y'_1,\dots,Y'_g\}\Rangle;$$
then we have a well-defined homomorphism
$$\vartheta:\bfa\to \cA,\qquad 
X_i\mapsto X'_i X''_i,\quad Y_i\mapsto Y'_i Y''_i.$$
By considering natural monomial bases in $\bfa$ and $\cA$ it is easy to verify that $\vartheta$ is an embedding. Since $\cA$ has a skew field of fractions by Theorem \ref{t:div}, $\bfa$ also has a skew field of fractions. Moreover, we can describe it explicitly by looking at appropriate evaluations of rational expressions in $X$ and $Y$ over $\kk$. Let
$$\cO_n^g=\left(\mm{n}^2\right)^g\times \left(\mm{n}^2\right)^g,\qquad \cO^g=\bigcup_n \cO_n^g.$$
Then we can evaluate $r\in \rexy$ at a point
\begin{equation}\label{e:bfpoint}
(a,b)=(a'_1,a''_1,\dots, a'_g,a''_g,b'_1,b''_1,\dots, b'_g,b''_g)\in\cO_n^g
\end{equation}
by replacing $X_i$ with
$$a'_i\otimes \overbrace{I \otimes \cdots\otimes a''_i\otimes\cdots I}^g\otimes I$$
and $Y_i$ with
$$I\otimes \overbrace{I \otimes \cdots\otimes b''_i\otimes\cdots I}^g\otimes b'_i.$$
Such a {\bf bf-evaluation} of $r$ is denoted $r(a,b)^{\bfs}\in\mm{n^{g+2}}$. The set of equivalence classes of nondegenerate expressions with respect to bf-evaluations on $\cO^g$ (cf.~Remark \ref{r:nn}) then becomes a skew field of fractions of $\bfa$, which we denote $\bfsf$. Its elements are called {\bf bi-free rational functions}.

For $n\in\N$ let
$$\bfvar{n}=\left\{(a,b)\in \mm{n}^g\times\mm{n}^g\colon 
[a_i,b_j]=0\ \forall i\neq j \right\}$$
be the set of {\bf bi-free tuples} of $n\times n$ matrices. The next proposition will demonstrate that bi-free rational functions have well defined nc-evaluations on bi-free tuples of matrices.

Let $K$ be an algebraically closed field and $a\in \GL_m(K)$. By considering the Jordan decomposition of $a$ it is easy to see that there exists $\tilde{a}\in \GL_m(K)$ such that $a=\tilde{a}^2$ and $\tilde{a}$ is a polynomial in $a$. Let us say that such $\tilde{a}$ is a \emph{regular square root} of $a$.

\begin{prop}\label{p:bfvar}
Let $r\in\rexy$.
\begin{enumerate}[\rm(a)]
\item If $r$ nc-vanishes on $\bfvar{n^{g+2}}$, then it bf-vanishes or is bf-undefined 
on $\cO_n^g$.
\item If $r$ bf-vanishes on $\cO_n^g$, then it bf-vanishes or is bf-undefined 
on $\bfvar{n}$.
\end{enumerate}
\end{prop}

\begin{proof}
(a) Trivial.

(b) Let $z$ be a $2g$-tuple of $n^{g+2}\times n^{g+2}$ generic matrices. Then there exist $p,q\in \pxy$ such that
$$r(z)=p(z)q(z)^{-1}.$$
If $x',x'',y',y''$ are $g$-tuples of $n\times n$ generic matrices, then $p$ bf-vanishes on $(x,y)$ by assumption.

Choose an arbitrary point $(a,b)\in\bfvar{n}$. Let $K$ denote the algebraic closure of the field $\kk(t)$, where $t$ is an auxiliary symbol. Matrices $a_i+tI$ and $b_i+tI$ are invertible over $K$, so they have regular square roots $\tilde{a}_i,\tilde{b}_i\in \GL_n(K)$. Now consider the unital $\kk$-subalgebra $\cB\subseteq \Mat_{n^{g+2}}(K)$ generated by
$$a_i'=\tilde{a}_i\otimes \overbrace{I \otimes \cdots\otimes \tilde{a}_i\otimes\cdots I}^g\otimes I,
\qquad 
b_i'=I\otimes \overbrace{I \otimes \cdots\otimes \tilde{b}_i\otimes\cdots I}^g\otimes \tilde{b}_i.$$
Since $[a_i+tI,b_j+tI]=0$ for $i\neq j$, we also have $[\tilde{a}_i,\tilde{b}_j]=0$ for $i\neq j$, so the restriction of the $K$-linear map
$$\ell:\Mat_{n^{g+2}}(K)\cong \Mat_n(K)^{\otimes (g+2)}\to \Mat_n(K),\quad c_0\otimes\cdots\otimes c_{g+1}
\mapsto c_0\cdots c_{g+1}$$
to $\cB$ is a homomorphism of $\kk$-algebras. Therefore
$$p(a+tI,b+tI)=\ell|_{\cB}(p(a',b'))=\ell|_{\cB}(0)=0$$ 
and so $p(a,b)=0$. Hence $r(a,b)=0$ if $r$ is defined at $(a,b)$.
\end{proof}

\begin{cor}\label{c:last}
Let $r\in\rexy$; then $r$ is bf-defined and nonzero on $\cO^g$ if and only if $r$ is nc-defined and nonzero on $\bfvar{n}$ for some $n\in\N$.
\end{cor}

Bi-free rational functions were constructed using a particular embedding $\vartheta$ of the bi-free algebra into a multipartite free algebra. While $\bfsf$ is thus a skew field of fractions of $\bfa$, it is unclear whether it is universal. Nevertheless, Corollary \ref{c:last} indicates that $\bfsf$ is at least canonical with respect to all finite-dimensional representations of $\bfa$, as the latter are given precisely by the bi-free tuples of matrices.
\end{appendices}

%%%%%%%%%%%%%%%%%%%%%%%%%%%%%%%%%%%%%%%%%%%%%%%%%%%%%%%%%%%%%%%%%%%%%
%%%%%%%%%%%%%%%%%%%%%%%%%%%%%%%%%%%%%%%%%%%%%%%%%%%%%%%%%%%%%%%%%%%%%

%\pagebreak
%\linespread{1.05}

\end{document}